\documentclass[11pt, a4paper]{amsart}
\usepackage{amsthm, amssymb, mathtools, cite, graphicx, color,, subcaption}
\mathtoolsset{showonlyrefs}
\newtheorem{theorem}{Theorem}[section]
\newtheorem{lemma}[theorem]{Lemma}
\newtheorem{proposition}[theorem]{Proposition}

\newtheorem{definition}[theorem]{Definition}

\numberwithin{equation}{section}
\title{Domains of Discontinuity for Almost-Fuchsian Groups}
\author{Andrew Sanders}

\begin{document}

\address{Department of Mathematics, Statistics and Computer Science, University of Illinois at Chicago, Chicago, IL 60607 USA}
\email{andysan@uic.edu}
\thanks{Sanders gratefully acknowledges partial support from the National Science Foundation Postdoctoral Research Fellowship}


\date{October 23,2013}


\begin{abstract}
An almost-Fuchsian group $\Gamma<\mathrm{Isom}^{+}(\mathbb{H}^3)$ is a 
quasi-Fuchsi\-an group such that the quotient hyperbolic manifold $\mathbb{H}^3/\Gamma$ contains 
a closed incompressible minimal surface with principal curvatures contained in $(-1,1).$
We show that the domain of discontinuity of an almost-Fuchsian group contains many balls 
of a fixed spherical radius $R>0$ in $\mathbb{C}\cup \{\infty\} =\partial_{\infty}(\mathbb{H}^3).$ 
This yields a necessary condition for a quasi-Fuchsian group to be almost-Fuchsian which involves only
conformal geometry.  As an application, we prove that there are no doubly-degenerate geometric limits 
of almost-Fuchsian groups.   
\end{abstract}

\maketitle

\section{Introduction}

The systematic study of closed minimal surfaces in hyperbolic 3-manifolds began with the work of Uhlenbeck in the early 1980's \cite{UHL83}.  There, she identified a class of quasi-Fuchsian hyperbolic 3-manifolds, the almost-Fuchsian manifolds, which contain a unique closed, incompressible minimal surface which has principal curvatures in $(-1,1).$  The structure of almost-Fuchsian manifolds has been studied considerably by a number of authors \cite{GHW10}, \cite{HW13}.  In particular, the invariants arising from quasi-conformal Kleinian group theory (e.g. Hausdorff dimension of limit sets, distance between conformal boundary components) are controlled by the principal curvatures of the unique minimal surface.

Given an almost-Fuchsian manifold $M,$ this paper further explores the relationship between the geometry of the unique minimal surface and the conformal surface at infinity.  As a result we will show that there are no doubly-degenerate geometric limits of almost-Fuchsian groups.  This will be achieved through a careful study of the hyperbolic Gauss map from the minimal surface, which serves to communicate information from the minimal surface to the conformal surface at infinity.

We briefly summarize the strategy: if $\Gamma$ is the holonomy group of an almost-Fuchsian manifold $M=\mathbb{H}^3/\Gamma,$ consider the $\Gamma$-invariant minimal disk $\widetilde{\Sigma}\subset\mathbb{H}^3$ which projects to the unique closed, minimal surface in $M.$  We locate disks $D_{i}\subset \widetilde{\Sigma}$ which are very close to being totally geodesic; the hyperbolic Gauss map from these disks is nicely behaved and in particular satisfies a Koebe-type theorem ~\cite{AG85}, namely its image contains disks of bounded radius.  Epstein ~\cite{EPS86} has studied the hyperbolic Gauss map extensively.  In particular, we apply his work to show that the images of the $D_i$ under the hyperbolic Gauss map are contained in the domain of discontinuity for $\Gamma.$  These images form barriers for the limit set of $\Gamma.$  Since the limit set of a doubly-degenerate group is equal to $\partial_{\infty}(\mathbb{H}^3),$ these barriers form an obstruction to $\Gamma$ approaching a doubly-degenerate group.

 If $\Sigma$ is a surface immersed in some hyperbolic 3-manifold $M,$  the immersion induces a conformal structure $\sigma$ on $\Sigma$ which underlies the induced Riemannian metric $g.$  Provided the immersion is minimal, Hopf ~\cite{HOP54} showed that the second fundamental form \[B=B_{11}dx^2 +2B_{12}dxdy+B_{22}dy^2\]
appears as the real part of
\[\alpha=(B_{11}-iB_{12})\ dz^2,\]
which is a holomorphic quadratic differential on $(\Sigma, \sigma).$  The norm $\lVert \alpha \rVert_{g}$ measures how much $\Sigma$ bends inside of $M.$  In $\S$\ref{subsec: holquad}, we prove a Harnack inequality for $\lVert \alpha \rVert_{g}$ satisfying some bound $\lVert \alpha \rVert_{g} \leq K.$  First we show that the norm of a bounded holomorphic quadratic differential on the hyperbolic plane satisfies a Harnack inequality.  Then, we prove the induced metric $g$ is uniformly comparable to the hyperbolic metric in that conformal class.  Therefore, the growth of the principal curvatures of a minimal immersion is bounded; the surface cannot bend too much, too quickly.  The disks $D_i$ mentioned in the previous paragraph are obtained by taking balls around the zeros of $\alpha;$ the Harnack inequality ensures we may pick balls of a uniform radius.

In $\S$\ref{sub: thickregions}, we begin the study of the hyperbolic Gauss map: given an oriented surface $\widetilde{\Sigma}\subset\mathbb{H}^3$ in hyperbolic 3-space with oriented unit normal field $N,$ the pair of hyperbolic Gauss maps $\mathcal{G}^{\pm}:\widetilde{\Sigma}\rightarrow \partial_{\infty}(\mathbb{H}^3)$ are defined by recording the endpoint of the geodesic ray in the direction of $\pm N.$  We show that the images of the disks $D_i\subset \widetilde{\Sigma}$ under the hyperbolic Gauss map contain disks of a bounded radius in $\partial_{\infty}(\mathbb{H}^3).$  To achieve this, we utilize the generalization of the K\"{o}ebe $\frac{1}{4}$-theorem to quasiconformal maps due to Gehring and Astala \cite{AG85}.  

In $\S$\ref{sec: main} the above technical results are used to prove that the domain of discontinuity of an almost-Fuchsian group $\Gamma$ contains a disk (in fact many) of fixed radius in $\mathbb{C}.$  As we mentioned above, this is an application of work of Epstein who showed ~\cite{EPS86} that the hyperbolic Gauss map from the minimal $\Gamma$-invariant disk $\widetilde{\Sigma}$ is a (quasiconformal) diffeomorphism onto  the domain of discontinuity for $\Gamma.$  In particular, we obtain definite regions $R_i$ of $\partial_{\infty}(\mathbb{H}^3)$ into which the limit set of $\Gamma$ can not penetrate.  
Lastly in $\S$\ref{subsec: limits} we prove the promised theorem:
\begin{theorem} \label{thm: ddlimits}
There are no doubly-degenerate geometric limits of almost-Fuchsian groups.
\end{theorem}
There is a technical issue in the proof of the above theorem; in order to estimate the size of the regions $R_i$ into which the limit set cannot penetrate, we must conjugate the group $\Gamma$ by some element of $\mathrm{Isom}(\mathbb{H}^3)$ to put the surface $\widetilde{\Sigma}$ into a normalized position.  It is possible that given a sequence of almost-Fuchsian groups $\Gamma_n,$ the elements $I_n$ by which we conjugate leave every compact set of $\mathrm{Isom}(\mathbb{H}^3).$  Since the action of $\mathrm{Isom}(\mathbb{H}^3)$ on $\partial_{\infty}(\mathbb{H}^3)$ is not isometric, this could destroy any control we had gained on the size of the $R_i.$  Proving that we may always conjugate the group into a normalized position by an isometry with some bounded translation distance relies strongly on the geometry of the minimal surface.
\subsection{Historical context}

In \cite{UHL83}, Uhlenbeck conjectured\footnote{It is more appropriate to say that this conjecture is implicit in the work of Uhlenbeck.  The exact quote in ~\cite{UHL83} is, "If the area minimizing surfaces in M are isolated, we expect a large number of minimal surfaces in quasi-Fuchsian manifolds near M."  The $M$ she refers to is the doubly-degenerate manifold which arises as a cyclic cover of a closed manifold fibering over the circle.} that any dou\-bly-degenerate hyperbolic 3-manifold M contains infinitely many distinct incompressible (stable) minimal surfaces which are homotopy equivalent to M.  Examples of this phenomenon are provided by the doubly-degenerate manifolds M which arise as cyclic covers of closed hyperbolic 3-manifolds fibering over the circle.  In this case, there is an infinite cyclic group of isometries of M.  The general existence theory (~\cite{FHS83}, ~\cite{SU82} and ~\cite{SY79}) yields an incompressible minimal surface $\Sigma$ in the closed manifold which M covers.  $\Sigma$ lifts to an incompressible minimal surface in M.  The translates of $\Sigma$ by the infinite cyclic group of deck transformations yields infinitely many distinct incompressible minimal surfaces in M.  Theorem 1.1 supports a general philosophy that the number of closed incompressible (stable) minimal surfaces in a quasi-Fuchsian manifold  should serve as a measure of how far that manifold is from being Fuchsian.  

In \cite{HW13}, the authors show that if $M=\mathbb{H}^3/\Gamma$ is an almost-Fuchsian manifold and $\lambda$ is the maximum positive principal curvature of the unique closed minimal surface in $M,$ then the Hausdorff dimension of the limit set of $\Gamma$ is at most $1+\lambda^2.$  The question remains whether there exists a sequence of almost-Fuchsian groups $\Gamma_n$ such that the Hausdorff dimension of the limit set approaches $2.$  Theorem \ref{thm: ddlimits} rules out the most naive way in which this might occur.

With these more technical comments in place, we comment on the main achievement of the paper.  The primary advance in this paper is a soft qualitative necessary condition for a quasi-Fuchsian group $\Gamma$ to be almost-Fuchsian.  The recipe is as follows: look at the domain of discontinuity.  For some small threshold $R>0,$ which we calibrate in the paper, if the domain of discontinuity does not contain many balls of radius larger than $R$ as measured in the spherical metric on $\partial_{\infty}(\mathbb{H}^3),$ then $\Gamma$ is not almost-Fuchsian.  Since almost-Fuchsian groups are defined directly using the hyperbolic geometry of the quasi-Fuchsian manifold, it is quite interesting that there is a necessary condition which can be stated solely in terms of the complex geometry of M\"{o}bius transformations.

\subsection{Plan of paper}

The preliminary section $\S$\ref{sec: prelim} assembles all of the definitions and basic facts we require from the theory of Kleinian groups, minimal surfaces and introduces the hyperbolic Gauss map.  Then in $\S$\ref{subsec: holquad} we prove the main technical propositions controlling the growth of the principal curvatures of a minimal surface in $\mathbb{H}^3.$  In $\S$\ref{sub: thickregions} we prove that under the hyperbolic Gauss map nearly flat regions in the minimal surface are sent to disks of controlled radius in the image of the Gauss map.  Finally, in $\S$\ref{sec: main} we prove that the domain of discontinuity of an almost-Fuchsian group contains disks of a fixed radius.  As an application we show that there are no doubly-degenerate geometric limits of almost-Fuchsian groups.

\subsection{Acknowledgments}  

This work contains part of the author's Ph.D. thesis which was completed under the supervision of Dr. William Goldman.  I am indebted to his friendship and mathematical encouragement over the years, without which this work and so much more may not have happened.  Additionally, Zheng Huang was instrumental in the development of this project; in particular he pointed out that special attention should be paid to the zeros of holomorphic quadratic differentials, and also highlighted the work of Epstein on surfaces in $\mathbb{H}^3,$ both of which turned out to be essential.

\section{Preliminaries}\label{sec: prelim}  In all that follows, $\Sigma$ denotes a closed, connected, oriented, smooth surface of genus greater than $1.$

\subsection{Kleinian groups and hyperbolic geometry}

The n-dimensional hyperbolic space $\mathbb{H}^n$ is the unique complete, 1-connected Riemannian manifold with constant negative sectional curvature $-1.$  We will primarily work in the upper half-space model for hyperbolic space which consists of the smooth manifold 

\[\mathbb{H}^n=\{(x_1,...,x_n)\in \mathbb{R}^n \lvert x_n>0\} \]

together with the Riemannian metric 

\[dh_n^2=\frac{\delta_{ij}dx^idx^j}{x_n^2}.\]
For $\mathbb{H}^3,$ we use coordinates $(z,t)\in\mathbb{H}^3$ with $z=x+iy\in\mathbb{C}$ and $t>0.$
A Kleinian group is a discrete (torsion-free) subgroup $\Gamma<\mathrm{Isom}^{+}(\mathbb{H}^3)\simeq \mathrm{PSL}(2,\mathbb{C})$ of orientation-preserving isometries of hyperbolic 3-space.  Given a Kleinian group $\Gamma,$ the action on $\mathbb{H}^3$ extends to an action on the conformal boundary $\partial_{\infty} (\mathbb{H}^3)\simeq \mathbb{C}\cup \{\infty\}$ by M\"{o}bius transformations.  This action divides $\partial_{\infty}(\mathbb{H}^3)$ into two disjoint subsets: $\Lambda(\Gamma)$ and $\Omega(\Gamma).$  The \textit{limit set} $\Lambda(\Gamma)$ is defined to be the smallest non-empty, $\Gamma$-invariant closed subset of $\partial_{\infty}(\mathbb{H}^3).$  The \textit{domain of discontinuity} $\partial_{\infty}(\mathbb{H}^3)\backslash \Lambda(\Gamma)=\Omega(\Gamma)$ is the largest open set on which $\Gamma$ acts properly discontinuously.  The quotient $M=\mathbb{H}^3/\Gamma$ is a complete hyperbolic 3-manifold with \textit{holonomy} group $\Gamma.$  

Given a sequence of Kleinian groups $\Gamma_n,$ we say that $\Gamma_n$ converges \textit{geometrically} to the group $\Gamma$ if,

\begin{enumerate}
\item For all $\gamma\in \Gamma$ there exists $\gamma_n\in\Gamma_n$ such that $\gamma_n\rightarrow\gamma$ in $\mathrm{Isom}(\mathbb{H}^3).$  
\item If $\gamma_n\in\Gamma_n$ and $\gamma_{n_j}\rightarrow \gamma,$ then $\gamma\in\Gamma.$
\end{enumerate}
We denote geometric convergence by $\Gamma_n\rightarrow \Gamma.$  It is well known ~\cite{CEG87} that the geometric convergence of Kleinian groups is equivalent to the base-framed, Gromov-Hausdorff convergence of the associated quotient manifolds.  

Equipping $\partial_{\infty}(\mathbb{H}^3)=\mathbb{C}\cup \{\infty\}$ with the spherical metric $d_{\mathbb{S}^2}=\frac{\lvert dz\rvert^2}{(1+\lvert z\rvert^2)^2},$ a sequence of closed subsets $A_n\subset \mathbb{C}\cup \{\infty\}$ converges to $A\subset \mathbb{C}\cup \{\infty\}$ in the \textit{Hausdorff} topology if $A_n\rightarrow A$ with respect to the distance
\[d(A_n,A)=\inf\left\{r: A_n\subset \bigcup\limits_{x\in A} B_{\mathbb{S}^2}(x,r)\ \mathrm{and} \ A\subset\bigcup\limits_{x\in A_n} B_{\mathbb{S}^2}(x,r)\right\}.\]
If a sequence of Kleinian groups $\Gamma_n$ converges geometrically to $\Gamma,$ it is a simple exercise to verify that the limit set of $\Gamma$ satisfies $\Lambda(\Gamma)\subset \lim\limits_{n\rightarrow\infty}\Lambda(\Gamma_n)$ where the limit is taken with respect to the Hausdorff topology. 

Now we restrict to the case that there is an isomorphism $\rho:\pi_1(\Sigma,p_0)\rightarrow \Gamma.$  The representation $\rho$ is \textit{quasi-Fuchsian} if and only if $\Omega(\Gamma)$ consists of precisely two invariant, connected, simply-connected components.  The quotient $\Omega(\Gamma)/\Gamma=X^+ \cup \overline{X^-}$ is a disjoint union of two marked Riemann surfaces $(X^+,\overline{X^-}),$ each diffeomorphic to $\Sigma,$ where the bar over $X^-$ denotes the surface with the opposite orientation.  The marking, which is a choice of homotopy equivalence $f^{\pm}:\Sigma\rightarrow X^{\pm},$ is determined by the requirement that $f_{*}^{\pm}=\rho.$  Conversely, we have the Bers' simultaneous uniformization theorem ~\cite{BER60},

\begin{theorem}
Given an ordered pair of marked, closed Riemann surfaces $(X^+,\overline{X^-})$ each diffeomorphic to $\Sigma,$  there exists an isomorphism $\rho:\pi_1(\Sigma,p)\rightarrow \Gamma,$ unique up to conjugation in $\mathrm{PSL}(2,\mathbb{C}),$ such that $\Omega(\Gamma)/\Gamma=X^+\cup \overline{X^-}.$  
\end{theorem}

It is a striking feature of the diversity of hyperbolic 3-manifolds that there exist isomorphisms $\rho:\pi_1(\Sigma,p)\rightarrow \Gamma$ onto Kleinian groups which are not quasi-Fuchsian.  A Kleinian surface group $\Gamma\simeq \pi_1(\Sigma,p)$ is called \textit{doubly degenerate} if the domain of discontinuity for $\Gamma$ is empty; $\Omega(\Gamma)=\emptyset.$  The first explicit examples of doubly-degenerate groups were free groups of rank 2 discovered by Jorgenson ~\cite{JOR77}.

Due to the resolution of the ending lamination conjecture, the isometry classification of doubly-degenerate manifolds is now well understood.  Nonetheless, their fine scale geometry is extremely intricate and interesting.  Bonahon showed ~\cite{BON86} that they are diffeomorphic to $\Sigma \times \mathbb{R},$ then the ending lamination theorem (for surface groups) ~\cite{MIN10}, ~\cite{BCM12} in conjunction with Thurston's double limit theorem ~\cite{THU} can be applied to see that every doubly degenerate group is a geometric limit of quasi-Fuchsian groups.  

\subsection{Minimal surfaces and almost-Fuchsian manifolds}\label{sub: minsurfaces}

Let $(M,h)$ be a 3-dimensional Riemannian manifold.  Given an immersion $f:\Sigma\rightarrow M,$ the area of 
$f$ is the area of the Riemannian manifold $(\Sigma, g)$ where $g=f^*(h).$  Suppose $f(\Sigma)$ is two sided in $M$ and let $\nu$ be a globally defined unit normal vector field.  Then given $X$ and $Y$ tangent vectors to $f(\Sigma),$ we have the second fundamental form B defined by,

\[B(X,Y)=h(\nabla_{X} \nu, Y)\]

where $\nabla$ is the Levi-Cevita connection of $h.$  The second fundamental form $B$ defines a symmetric, contravariant 2-tensor on $\Sigma.$  By the spectral theorem for self-adjoint operators, $B$ is diagonalizable with eigenvalues $\lambda_1, \lambda_2$ whose product $\lambda_1\lambda_2$ is a smooth function on $\Sigma.$  They are called the \textit{principal curvatures} of the immersion and $\lVert B\rVert_{g}^2=\lambda_1^2+\lambda_2^2.$   

 Contracting with the metric $g$ yields the mean curvature of the immersion,

\[H=g^{ij}B_{ij}\]
where repeated upper and lower indices are to be summed over.  The immersion $f:\Sigma\rightarrow M$ is \textit{minimal} if $H=\lambda_1+\lambda_2=0.$  Equivalently, $f$ is a critical point for the area functional defined on the space of immersions of $\Sigma$ into $M.$  A minimal surface is said to be \textit{least area} if $f:\Sigma\rightarrow M$ has least area among all maps homotopic to $f;$ least area maps are minimal.

Now suppose $M$ is a hyperbolic 3-manifold.  If $f:\Sigma\rightarrow M$ is an immersion as above, the Gauss equation of the immersion is
\begin{align} \label{eqn:gauss1}
K_{g}=-1+\lambda_1\lambda_2 
\end{align}  
where $K_g$ is the sectional curvature of the metric $g.$

If $f$ is a minimal immersion, it follows that 
\begin{align}
K_{g}=-1-\lambda^2
\end{align}
where we have packaged the principal curvatures into a single smooth function $\lambda^2=\lambda_1^2=\lambda_2^2$ since they satisfy $\lambda_1=-\lambda_2.$  Then \begin{align}
\frac{1}{2} \lVert B \rVert_{g}^2=\lambda^2
\end{align}
and 
\begin{align}
K_{g}=-1-\frac{1}{2} \lVert B \rVert_{g}^2.
\end{align}

We now present the principal object of study. 

\begin{definition}
A quasi-Fuchsian 3-manifold $M$ is almost-Fuchsian if there exists a minimal immersion $f:\Sigma\rightarrow M$ with $\lVert B \rVert_{g}^2(x) <2$ for all $x\in\Sigma.$
\end{definition}

The following classification theorem is due to Uhlenbeck ~\cite{UHL83}.

\begin{theorem}
\label{thm:UHAF}
Suppose $M$ is almost-Fuchsian and homotopy equivalent to $\Sigma.$  Then,
\begin{itemize}
\item $f:\Sigma\rightarrow M$ such that $\lVert B \rVert_{g}^2(x) <2$ is the only closed minimal surface of any kind in M. \\
\item $f:\Sigma\rightarrow M$ is incompressible, i.e. it is a $\pi_1-$injective, smooth embedding.  \\
\item The normal exponential map from $f(\Sigma)$ is a diffeomorphism onto M.  In these coordinates the metric on $\Sigma \times \mathbb{R}$ can be expressed as
\[dt^2 + g(\cosh{(t)}\ \mathbb{I}(\cdot)+\sinh{(t)}\ \mathbb{S}(\cdot),\cosh{(t)}\ \mathbb{I}(\cdot)+\sinh{(t)}\ \mathbb{S}(\cdot)).\]  
Here $\mathbb{I}$ is the identity operator and $\mathbb{S}(X)=\nabla_X(\nu)$ is the shape operator associated to the second fundamental form.
\end{itemize}
\end{theorem}

The following theorem, first observed by Hopf ~\cite{HOP54}, will be of great importance to us.

\begin{theorem}
Let $f:\Sigma\rightarrow M$ be a minimal immersion and M a constant curvature Riemannian 3-manifold.  Let $z$ be a local holomorphic coordinate for the Riemann surface structure induced on $\Sigma.$  If $B$ is the second fundamental form of the immersion, then the expression

\[\alpha = (B_{11}-iB_{12})dz^2\]
is an invariantly defined holomorphic quadratic differential on $\Sigma$ with $\mathfrak{Re}(\alpha)=B.$
\end{theorem} 

In terms of the quadratic differential $\alpha,$ the almost-Fuchsian condition translates to the condition that $\lambda^2=\lVert \alpha \rVert_{g}^2< 1.$  Uhlenbeck's theorem \ref{thm:UHAF} along with an implicit function theorem argument (see ~\cite{UHL83}) shows that the space of isotopy classes of almost-Fuchsian manifolds is parameterized via an open neighborhood of the zero section of the bundle of holomorphic quadratic differential over $\mathcal{T}(\Sigma)$ the Teichmuller space of $\Sigma.$  Recall the Teichmuller space $\mathcal{T}(\Sigma)$ is the space of isotopy classes of complex structures on $\Sigma$ compatible with the orientation.  The bundle of holomorphic quadratic differentials can be identified with $T^{*}\mathcal{T} (\Sigma)$ the cotangent bundle of $\mathcal{T}(\Sigma).$  

We will denote the space (of isotopy classes) of almost-Fuchsian manifolds by $\mathcal{AF},$ and say that $(g,\alpha)\in\mathcal{AF}$ is an almost-Fuchsian manifold with data $(g,\alpha)$ if there exists an almost-Fuchsian manifold such that the unique minimal surface has induced metric $g$ and second fundamental form $\mathfrak{Re}(\alpha).$  Throughout the paper, we are concerned with specific manifolds and not just manifolds up to isotopy so $(g,\alpha)\in\mathcal{AF}$ is understood to be a representative of the isotopy class. 

Given $(g,\alpha)\in\mathcal{AF}$ with holonomy group $\Gamma,$ we call the unique embedded $\Gamma$-invariant minimal disk $\widetilde{\Sigma}\subset\mathbb{H}^3$ an almost-Fuchsian disk with data $(g,\alpha).$  An almost-Fuchsian disk $\widetilde{\Sigma}$ is \textit{normalized} if
\begin{enumerate}
\item The point $p=(0,0,1)\in \widetilde{\Sigma}$ and the oriented unit normal to $\widetilde{\Sigma}$ at $p$ is $(0,0,-1).$
\item The principal curvatures vanish at $p:\ \alpha(p)=0.$ 
\end{enumerate}

If $(g,\alpha)\in\mathcal{AF}$ with holonomy group $\Gamma,$ we can always select $I\in \mathrm{Isom}(\mathbb{H}^3)$ such that the $I\Gamma I^{-1}$-invariant almost-Fuchsian disk is normalized.  Such a choice of $I$ is unique up to the action of the circle group $U(1)<\mathrm{Isom}(\mathbb{H}^3)$ of rotations about the z-axis in the upper half-space model of hyperbolic space.

\subsection{Hyperbolic Gauss Map}
Details on the material in this section may be found in the paper of Epstein ~\cite{EPS86}.  Let $S\subset\mathbb{H}^3$ be an oriented, embedded surface.  Let $N$ be a global unit normal vector field on $S$ such that if $\{X,Y\}\subset TS$ is an oriented basis of the tangent space of $S,$ then $\{X,Y, N\}$ extends to an oriented basis of $T\mathbb{H}^3.$  Given $p\in S,$ let $\gamma_p(t)$ be the unit speed geodesic ray with
initial point $\gamma_p(0)=p$ and initial velocity $\frac{d\gamma_p(t)}{dt}\vert_{t=0}=N.$

\begin{definition}
The forward hyperbolic Gauss map associated to $S$ is the map $\mathcal{G}_{S}^{+}:S\rightarrow \partial_{\infty}(\mathbb{H}^3)$ defined by
\[\mathcal{G}_{S}^{+}(p)=\displaystyle\lim_{t\rightarrow+\infty}\gamma_p(t).\]
\end{definition} 

If we use $-N$ in the definition we shall call the associated map $\mathcal{G}_{S}^{-}$ the backwards hyperbolic Gauss map.  

We quickly recall the results we shall utilize from ~\cite{EPS84}, ~\cite{EPS86}.  Let $f:\mathbb{D}\rightarrow \mathbb{H}^3$ be an immersion of the unit disk $\mathbb{D}$ such that the principal curvatures are always contained in $(-1+\varepsilon,1-\varepsilon)$ for some $\varepsilon>0.$  Epstein proves,
\begin{theorem}[\cite{EPS84}, ~\cite{EPS86}] \label{thm:eps}
Let $S=f(\mathbb{D}).$  The immersion $f:\mathbb{D}\rightarrow \mathbb{H}^3$ satisfies,
\begin{enumerate}
\item $f$ is a proper embedding. 
\item $f$ extends to an embedding $\overline{f}:\overline{\mathbb{D}}\rightarrow \mathbb{H}^3\cup \partial_{\infty}(\mathbb{H}^3).$  So $\partial_{\infty}(S)=\overline{f}(\partial \mathbb{D}\setminus \mathbb{D})$ is a Jordan curve.
\item Each hyperbolic Gauss map $\mathcal{G}_{S}^{\pm}:S\rightarrow \partial_{\infty}(\mathbb{H}^3)$ is a quasiconformal diffeomorphism onto a component of $\partial_{\infty}(\mathbb{H}^3)\setminus 
\partial_{\infty}(S).$  Note that the Jordan curve theorem implies $\partial_{\infty}(\mathbb{H}^3)\setminus \partial_{\infty}(S)$ consists of two components.  
\end{enumerate}
\end{theorem}
\begin{figure}[h]
  \centering
    \includegraphics[width=3in]{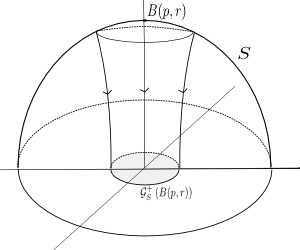}
  \caption{The image of a disk under the forward hyperbolic Gauss map from a totally geodesic plane in $\mathbb{H}^3.$}
\end{figure}

Recall that given a homeomorphism $f:D_1\rightarrow D_2$ between domains $D_1, D_2\subset \mathbb{C},$ we say that $f$ is a \textit{quasiconformal} homeomorphism if:
\begin{itemize}
\item $f\in H^1_{loc}(D_1),$ that is both $f$ and its distributional derivatives $f_z,\ f_{\overline{z}}$ are locally square-integrable on $D_1.$  
\item There exists $\mu\in L^{\infty}(D_1)$ with $\lVert \mu \rVert_{L^{\infty}}<1$ such that $f_{\overline{z}}=\mu f_z$ in the sense of distributions.  
\end{itemize}
\textbf{Remark:}  The quasiconformal maps we meet will be smooth, so the references to distributions above can be safely ignored.

Defining the \textit{dilatation} of $f$ to be the quantity 
\begin{align}
K=\frac{1+\lVert \mu \rVert_{L^{\infty}}}{1-\lVert \mu \rVert_{L^{\infty}}},\ 
\end{align}
$f$ is said to be $K$-quasiconformal.  Observe that a $1$-quasiconformal map is conformal.

From theorem \ref{thm:eps} it immediately follows that if $\widetilde{\Sigma}\subset \mathbb{H}^3$ is an almost-Fuchsian disk invariant under an almost-Fuchsian group $\Gamma,$ each hyperbolic Gauss map
\[\mathcal{G}_{\widetilde{\Sigma}}^{\pm}:\widetilde{\Sigma}\rightarrow \partial_{\infty}(\mathbb{H}^3)\]
is a (quasiconformal) diffeomorphism onto one component of the domain of discontinuity $\Omega$ of $\Gamma.$  Furthermore, $\pi_1(\Sigma, p_0)\simeq \Gamma$ implies
$\partial_{\infty}(\widetilde{\Sigma})$ identifies homeomorphically with the limit set $\Lambda(\Gamma).$

\section{Technical Estimates}\label{subsec: holquad}

This section contains some technical results which give us control on the growth of the principal curvature function for a minimal surface immersed in a hyperbolic 3-manifold.

We begin with a basic result establishing a bound on the growth of the $L^{\infty}$-norm of a bounded, holomorphic quadratic differential $\alpha=f(z) \ dz^2$ on $\mathbb{H}^2.$  We denote the canonical bundle of holomorphic 1-forms on $\mathbb{H}^2$ by $K_{\mathbb{H}^2}.$  In this section we will utilize the Poincar\'{e} disk model of the hyperbolic plane consisting of the unit disk $\mathbb{D}\subset \mathbb{C}$ with the metric
\[\frac{4\lvert dz\rvert^2}{(1-\lvert z\rvert^2)^2}.\]
 
\begin{proposition}
\label{prop:bdQD}
Let $\alpha\in H^{0}\left(\mathbb{H}^2,K_{\mathbb{H}^2}^2\right)$ be a holomorphic quadratic differential on the hyperbolic plane and assume there exists a $C>0$ such that $\sup\limits_{z\in\mathbb{H}^2}\lVert\alpha\lVert_{\mathbb{H}^2}(z)\leq C.$  Assume there is $z_0\in \mathbb{H}^2$ such that $\alpha(z_0)=0.$  Then for all $\varepsilon>0,$ there exists $r=r(\varepsilon, C)>0$ such that 
\[\lVert\alpha\rVert_{\mathbb{H}^2}(x)<\varepsilon\]
for all $x\in B_{\mathbb{H}^2}(z_0,r)$ where $B_{\mathbb{H}^2}(z_0,r)$ is the ball of hyperbolic radius $r$ centered at $z_0.$
\end{proposition} 

\begin{proof}
For this proof we work in the Poincar\'{e} disk model of hyperbolic space consisting of the unit disk $\mathbb{D}\subset \mathbb{C}$ with the metric
\begin{align}
\frac{4\lvert dz\rvert^2}{(1-\lvert z\rvert^2)^2}.
\end{align}
Without loss of generality, assume $z_0=0\in \mathbb{D}.$  Writing $\alpha(z)=f(z)\ dz^2,$ the condition
\begin{align} 
\sup\limits_{z\in\mathbb{H}^2}\lVert\alpha\lVert_{\mathbb{H}^2}(z)\leq C
\end{align}
becomes
\begin{align}\label{eqn: normbound}
\frac{(1-\lvert z\rvert^2)^2}{4}\lvert f(z)\rvert\leq C
\end{align}
for all $z\in\mathbb{D}.$  Define an auxiliary holomorphic function defined on $\mathbb{D}$ by
\begin{align}
g(z)=C'f\left(\frac{1}{2}z\right)
\end{align}
with 
\begin{align}
C'=\frac{\left(1-\frac{1}{4}\right)^2}{4C}.
\end{align}
Utilizing \eqref{eqn: normbound}, 
\begin{align}
\lvert g(z)\rvert=\frac{\left(1-\frac{1}{4}\right)^2}{4C}\left\lvert f\left(\frac{1}{2}z\right)\right\rvert\leq \frac{\left(1-\frac{1}{4}\right)^2}{\left(1-\frac{\lvert z\rvert^2}{4}\right)^2}\leq 1
\end{align}
which proves that 
\[\lvert g(z) \rvert \leq 1.\]
By construction $g(0)=0,$ hence the Schwarz lemma implies
\begin{align}
\lvert g(z) \rvert\leq \lvert z\rvert.
\end{align}
Thus,
\begin{align}
\lVert \alpha \rVert_{\mathbb{H}^2}\left(\frac{1}{2}z\right)=\frac{(1-\frac{1}{4}\lvert z\rvert^2)^2}{4}\left\lvert f\left(\frac{1}{2}z\right)\right\rvert
\leq \frac{(1-\frac{1}{4}\lvert z\rvert^2)^2}{4C'}\lvert z\rvert<\varepsilon
\end{align}
provided $\lvert z\rvert <4C'\varepsilon.$
Thus, on the Euclidean disk of radius
$2C'\varepsilon$ the estimate stated in the theorem holds.  Translating to hyperbolic radius we obtain
\begin{align}
r=\log\left(\frac{1+2C'\varepsilon}{1-2C'\varepsilon}\right).
\end{align}
This completes the proof.
\end{proof}
Our next task is to port the bounds achieved in Proposition \ref{prop:bdQD} to the case of the norm of a quadratic differential arising from a minimal surface in an almost-Fuchsian manifold.  Fortunately,
the next lemma shows that the induced metric on the surface is uniformly comparable to the hyperbolic metric.  Recall that given a Riemannian metric $g$ on $\Sigma,$ the K\"{o}ebe-Poincar\'{e} uniformization theorem provides a unique
hyperbolic metric in the conformal class of $g.$ 

\begin{lemma}
\label{lem:AFmetriccomp}
Let $(g,\alpha)\in\mathcal{AF}$ and suppose $h$ is the unique hyperbolic metric in the conformal class of $g$ with $g=e^{2u}h.$  Then,
\[\frac{-\ln(2)}{2}< u \leq 0.\]
\end{lemma}

\begin{proof}
Since $g=e^{2u}h$ and the sectional curvature of $h$ is equal to $-1,$ we have the standard equation,

\[K_g=e^{-2u}(-\Delta_h u-1)\]
where $K_g$ is the sectional curvature of $g.$  A further application of the Gauss equation \eqref{eqn:gauss1} reveals,
\begin{align}
-1-\lVert\alpha\rVert_{g}^2=e^{-2u}(-\Delta_h u-1). \label{eqn:gauss2}
\end{align}

Since $(g,\alpha)\in\mathcal{AF},$ the inequality $-1-\lVert\alpha\rVert_{g}^2> -2$ holds.  Let $p\in\Sigma$ be a local minimum for $u.$  Then,

\[-2< e^{-2u}(-\Delta_h u-1)\leq -e^{-2u}\]
since $-\Delta_{h}u(p)\leq 0.$  Thus, $\frac{-\ln(2)}{2}< u.$

Now we apply the maximum principal directly to the equation \eqref{eqn:gauss2}, 
\[\Delta_{h}u+1-e^{2u}-e^{-2u}\lVert\alpha\rVert_{h}^2=0.\]
If $p\in\Sigma$ is a local maximum of $u,$ then $\Delta_{h} u(p)\leq 0$ which implies $-e^{2u(p)}+1\geq 0,$ so $u\leq 0$ and the proof is complete.
\end{proof}

In lieu of the previous lemma, the following proposition is a direct consequence of proposition \ref{prop:bdQD}.
\begin{proposition}
\label{prop:bdQD2}
Let $(g,\alpha)\in\mathcal{AF}$ and $p\in\Sigma$ such that  $\alpha(p)=0.$  Then for all $\varepsilon>0,$ there exists $r(\varepsilon)>0$ such that 
\[\lVert\alpha\rVert_{g}(x)<\varepsilon\]
for all $x\in B_{g}(p,r),$ where $B_{g}(p,r)$ is the ball of radius $r$ centered at $p$ as measured in the metric $g.$
\end{proposition}

\textbf{Remark:} Note that $\alpha$ is a holomorphic quadratic differential on a closed Riemann surface of genus $g>1.$  Thus, they are holomorphic sections of the square of the canonical bundle $K^2$ which has degree $4g-4.$  By standard Riemann surface theory, $\alpha$ has $4g-4$ zeros counting multiplicity.   
\\
\\
\textbf{Remark:} The above proposition is equally true if we work in the metric universal cover $\widetilde{\Sigma}$ of $\Sigma.$  This will be important in our applications.

\begin{proof}
Let $g=e^{2u}h$ where $h$ is the unique hyperbolic metric in the conformal class of $g.$  Then $\lVert \alpha \rVert_{g}=e^{-2u}\lVert \alpha \rVert_{h}.$
Combined with Lemma ~\ref{lem:AFmetriccomp} this implies $\lVert \alpha \rVert_{g}\leq 2\lVert \alpha \rVert_{h}.$  Let $p\in\Sigma$ be such that $\alpha(p)=0.$  Since $\lVert \alpha \rVert_{g}<1,$ proposition ~\ref{prop:bdQD} implies that for all $\varepsilon>0$ we can find $r'>0$ such that 
$\lVert \alpha \lVert_{h}(x)<\frac{\varepsilon}{2}$ for all $x\in B_{h}(p,r').$  Now, since $u> \frac{-\ln(2)}{2}$ we observe $B_{g}(p,r)\subset B_{h}(p,r')$ where
$r=\frac{r'}{\sqrt{2}}.$ It follows that 
\[\lVert \alpha \rVert_{g}(x)\leq 2\lVert \alpha \rVert_{h}(x)\leq \varepsilon\]
for all $x\in B_{g}(p,r)\subset B_{h}(p,r').$
\end{proof}

By the Gauss equation, on $B_{g}(p,r)$ the sectional curvature $K_g$ of $g$ satisfies $-1-\varepsilon^2 \leq K_g \leq -1.$  Thus the disks $B_g(p,r)$ are very close to being totally geodesic. These are the nearly geodesic regions referred to in the introduction.

\section{Thick regions in the Gauss map image} \label{sub: thickregions}

In this section we show that the nearly geodesic regions previously obtained are mapped, via the hyperbolic Gauss map, to regions in $\partial_{\infty}(\mathbb{H}^3)$ which have a uniformly bounded thickness.  This will rely on the generalization of the K\"{o}ebe $\frac{1}{4}$-theorem due to Gehring and Astala \cite{AG85}.

Let $(g,\alpha)$ be the data for a normalized almost-Fuchsian disk $\widetilde{\Sigma}\subset \mathbb{H}^3.$  By the uniformization theorem, there exists a conformal diffeomorphism
\begin{align}
\phi: \mathbb{D}\rightarrow (\widetilde{\Sigma},g)
\end{align}
satisfying $f(0)=p.$  Since $\phi$ is conformal,
\begin{align}
\phi^* g=e^{2u}h
\end{align}
for some $u\in C^{\infty}(\Sigma)$ where $h$ is the hyperbolic metric on $\mathbb{D}.$  If $dV_g$ is the volume element arising from the metric $g,$ then
\begin{align}\label{eqn: jacobian}
\phi^{*}dV_g=e^{2u}dV_h.
\end{align}

\textbf{Convention:}  The following \textbf{notation} is in place for the rest of this section:  $\widetilde{\Sigma}\subset \mathbb{H}^3$ is a normalized almost-Fuchsian disk with data $(g,\alpha).$  Recall, this means $(0,0,1)=p\in\widetilde{\Sigma}$ is such that $\alpha(p)=0.$
Choose an $\varepsilon>0,$ then Proposition \ref{prop:bdQD2} yields $r>0$ such that the norm  of $\alpha$ is less than $\varepsilon$ on $B_{g}(p,r).$  Whenever $r$ or $\varepsilon$ appear in statements below, it is these fixed quantities to which we refer.

\begin{proposition}\label{thm: uniform}
Let $\widetilde{\Sigma}$ be a normalized almost-Fuchsian disk with data $(g,\alpha).$  Let
\begin{align}
\phi:\mathbb{D}\rightarrow (\widetilde{\Sigma},g)
\end{align}
be a uniformization such that $\phi(0)=p.$
Then there exists $r_1=r_1(r)>0$ such that
\begin{align}
B_{\mathbb{C}}(0,r_1)\subset \phi^{-1}(B_g(p,r)).
\end{align}
Furthermore,
\begin{align}
J_{\phi}(z)>2
\end{align}
for all $z\in \mathbb{D}.$  Here $J_{\phi}(z)$ is the Jacobian of $\phi$ at $z$ calculated with respect to the Euclidean metric on $\mathbb{D}.$
\end{proposition}

\begin{proof}
All of the conclusions are direct consequences of Lemma \ref{lem:AFmetriccomp}.  On a fixed compact subset of $\mathbb{D},$ the metric $g$ is uniformly comparable to the hyperbolic metric and the hyperbolic metric is uniformly comparable to the Euclidean metric, thus there exists $r_1>0$ such that
\begin{align}
B_{\mathbb{C}}(0,r_1)\subset \phi^{-1}(B_g(p,r)).
\end{align}
Writing $\phi^*(g)=e^{2u}h,$ then by \eqref{eqn: jacobian}, the Jacobian of $\phi$ (with respect to the Euclidean metric on $\mathbb{D}$) is given by
\begin{align}
J_{\phi}(z)=\frac{4e^{2u(z)}}{(1-\lvert z\rvert^2)^2}.
\end{align}
Applying the estimate in Lemma \ref{lem:AFmetriccomp} reveals
\begin{align}
J_{\phi}(z)> 2
\end{align}
for all $z\in \mathbb{D}=\mathbb{H}^2.$  This completes the proof.
\end{proof}

The following proposition which controls the distortion of the hyperbolic Gauss map is due to Epstein \cite{EPS86}.
\begin{proposition}\label{thm: jacgauss}
Let $\widetilde{\Sigma}$ be a normalized almost-Fuchsian disk with data $(g,\alpha).$  Then:
\begin{enumerate}
\item The hyperbolic Gauss map
\begin{align}
\mathcal{G}_{\widetilde{\Sigma}}^+:B_g(p,r)\rightarrow \mathbb{C}
\end{align}
is quasiconformal with dilatation $K\leq \left(\frac{1+\varepsilon}{1-\varepsilon}\right)^{\frac{1}{2}}.$
\item There exists a universal constant $C>0$ such that
\begin{align}
J_{\mathcal{G}_{\widetilde{\Sigma}}^{+}}(x)>C(1-\varepsilon^2) \label{eqn: jacJ}
\end{align}
for all $x\in B_g(p,r)$ where the Jacobian of $\mathcal{G}^+$ is computed with respect to the Euclidean metric on $\mathbb{C}.$
\end{enumerate}
\end{proposition}

\begin{proof}
These facts can be found on pages $120-121$ of \cite{EPS86}.  Epstein uses the spherical metric on $\mathbb{C},$ but in a bounded neighborhood of zero the spherical metric and Euclidean metric are uniformly
comparable.  Thus, the $C$ we obtain in \eqref{eqn: jacJ} is some multiple of that obtained by Epstein.
\end{proof}

Next, we introduce the generalization of K\"{o}ebe's $\frac{1}{4}$-theorem due to Gehring and Astala \cite{AG85}.  Let $U,V\subset \mathbb{C}$ be open domains and
\begin{align}
f:U\rightarrow V
\end{align}
a $K$-quasiconformal mapping with Jacobian $J_f.$  Then $\log J_f$ is locally integrable and the quantity
\begin{align}
(\log J_f)_{B}=\frac{1}{\lvert B \rvert}\int_{B} \log J_f \ dx,
\end{align}
with $B\subset U$ a ball, is well defined.  For each $x\in U,$ define
\begin{align}
B(x)=B_{\mathbb{C}}(x,d_{\mathbb{C}}(x,\partial U))
\end{align}
to be the largest open ball centered at $x$ which remains in $U.$  Lastly, define
\begin{align}
a_{f}(x)=\text{\textnormal{exp}}\left(\frac{1}{2}(\log J_f)_{B(x)}\right).
\end{align}
The promised generalization follows.
\begin{theorem}[\cite{AG85}] \label{thm: koebe}
Suppose $U$ and $V$ are open domains in $\mathbb{C}$ and $f:U\rightarrow V$ is $K$-quasiconformal. Then there exists a constant $C=C(K)$ such that
\begin{align}
\frac{1}{C}\frac{d_{\mathbb{C}}(f(x),\partial V)}{d_{\mathbb{C}}(x,\partial U)}\leq a_{f}(x)\leq C  \frac{d_{\mathbb{C}}(f(x),\partial V)}{d_{\mathbb{C}}(x,\partial U)}.
\end{align}
\end{theorem}

Given $\widetilde{\Sigma}\subset \mathbb{H}^3$ a normalized almost-Fuchsian disk with data $(g,\alpha),$ fix a uniformization
\begin{align}
\phi: \mathbb{D}\rightarrow (\widetilde{\Sigma},g)
\end{align}
such that $\phi(0)=p$ as in Proposition \ref{thm: uniform}.  Consider the composition
\begin{align}
\Phi:=\mathcal{G}_{\widetilde{\Sigma}}^{+}\circ \phi :\mathbb{D}\rightarrow \mathbb{C}
\end{align}
where we have identified $\partial_{\infty}(\mathbb{H}^3)$ with $\mathbb{C}\cup \{\infty\}.$  Our strategy is to apply Theorem \ref{thm: koebe} to the function $\Phi.$  We collect the necessary properties of $\Phi$ below.

\begin{proposition}\label{prop: applykoebe}
The map $\Phi:\mathbb{D}\rightarrow \mathbb{C}$ above satisfies:
\begin{enumerate}
\item $\Phi(0)=0$
\item The restriction of $\Phi$ to the sub-disk $B_{\mathbb{C}}(0,r_1)$ from Proposition \ref{thm: uniform} is $K$-quasiconformal with $K$ independent of $(g,\alpha).$
\item There exists $\beta(r_1)=\beta>0$ such that $J_{\Phi}(z)>\beta$ for all $z\in B_{\mathbb{C}}(0,r_1)$ with $\beta$ independent of $(g,\alpha).$
\end{enumerate}
\end{proposition}

\begin{proof}
First,
\[\Phi(0)=0\]
is a direct consequence of the fact that $\widetilde{\Sigma}$ is normalized.
Next, $\Phi$ is a composition of a conformal map $\phi$ with the hyperbolic Gauss map $\mathcal{G}_{\widetilde{\Sigma}}^{+}.$  Thus, the restriction of $\Phi$ to $B_{\mathbb{C}}(0,r_1)$ is K-quasiconformal if $\mathcal{G}_{\widetilde{\Sigma}}^{+}$ is K-quasiconformal on $B_{g}(p,r).$  This is proved in Proposition \ref{thm: jacgauss} which verifies $(2).$
Lastly, the Jacobian is multiplicative:
\begin{align}
J_{\Phi}(z)=J_{\mathcal{G}_{\widetilde{\Sigma}}^+}\left(\phi(z)\right)J_{\phi}(z).
\end{align}
Combining Propositions \ref{thm: uniform} and \ref{thm: jacgauss}, the product on the right is bounded below by $\beta=2C(1-\varepsilon^2).$ This proves $(3)$ and the proof is complete.

\end{proof}

Proposition \ref{prop: applykoebe} and Theorem \ref{thm: koebe} combine to show that the image of the Gauss map contains disks of a definite size.
\begin{proposition}
\label{prop:thickGaussMap} Let $\widetilde{\Sigma}$ be a normalized almost-Fuchsian disk.  Then there exists $R>0$ such that $B_{\mathbb{C}}(0,R)\subset\mathcal{G}_{\widetilde{\Sigma}}^+(\widetilde{\Sigma})$
where $B_{\mathbb{C}}(0,R)$ is the Euclidean disc of radius $R$ centered at $0\in\mathbb{C}.$
\end{proposition}

\begin{proof}
Consider the map
\begin{align}
\Phi:=\mathcal{G}_{\widetilde{\Sigma}}^+ \circ \phi : B_{\mathbb{C}}(0,r_1)\rightarrow \mathbb{C}.
\end{align}
By Proposition \ref{prop: applykoebe}$(3),$ there exists $\beta>0$ such that
\begin{align}
a_{\Phi}(0)\geq \text{\textnormal{exp}}\left(\frac{1}{2}\log \beta\right).
\end{align}
Noting that $\Phi$ is $K$-quasiconformal by Proposition \ref{prop: applykoebe}$(2),$ Theorem \ref{thm: koebe} yields
\begin{align}
\sqrt{\beta}\leq \frac{C}{r_1}d_{\mathbb{C}}\bigg(0,\partial \Big(\Phi\big(B_{\mathbb{C}}(0,r_1)\big)\Big)\bigg).
\end{align}
where we have taken $U=B_{\mathbb{C}}(0,r_1)$ and $V=\Phi(U).$
Taking $R=\frac{r_1\sqrt{\beta}}{C}$ completes the proof.
\end{proof}

\section{Main results}\label{sec: main}

In this section, we present the main results concerning the domain of discontinuity of an almost-Fuchsian group.  An almost-Fuchsian group $\Gamma$ is normalized if the $\Gamma$-invariant almost-Fuchsian disk is normalized.

\begin{theorem}
\label{thm:thickDD}
Let $\Gamma$ be a normalized almost-Fuchsian group.  Then the domain of discontinuity $\Omega$ of $\Gamma$ contains $B_{\mathbb{C}}(0,R)\subset \mathbb{C}$ for some $R>0.$
\end{theorem}

\begin{proof}
Let $\widetilde{\Sigma}$ be the normalized $\Gamma$-invariant almost-Fuchsian disk.  By Theorem \ref{thm:eps}, the forward hyperbolic Gauss map $\mathcal{G}_{\widetilde{\Sigma}}^+:\widetilde{\Sigma}\rightarrow \partial_{\infty}(\mathbb{H}^3)$
is a (quasi-conformal) diffeomorphism onto one connected component $\Omega^{+}$ of the domain of discontinuity.  By Proposition \ref{prop:thickGaussMap}, $B_{\mathbb{C}}(0,R)\subset\mathcal{G}_{\widetilde{\Sigma}}^+(\widetilde{\Sigma})=\Omega^{+}$ for some $R>0.$
\end{proof}

\begin{figure}[h]
  \centering
    \includegraphics[width=3in]{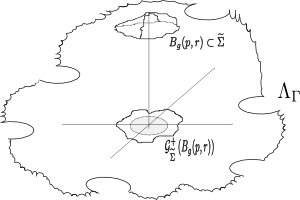}
  \caption{The disk contained inside the image of the hyperbolic Gauss map lying in the domain of discontinuity 
for some almost-Fuchsian group $\Gamma$ with limit set $\Lambda_{\Gamma}.$}
\end{figure}

Note that given an almost-Fuchsian group $\Gamma,$ conjugating $\Gamma$ by a rotation so that $0\in\Omega$ followed by a translation with arbitrarily large translation distance in the direction of the positive $t$-axis, we obtain an almost-Fuchsian group whose domain of discontinuity contains a disk around zero of arbitrarily large radius.  Of course, the resulting almost-Fuchsian group is not normalized.  This is the essential point of the above theorem.

We can actually do better than the previous theorem.  Because the almost-Fuchsian disk $\widetilde{\Sigma}$ is the universal cover of a closed minimal surface whose non-negative principal curvature equals the norm of a holomorphic quadratic differential, the principal curvatures of $\widetilde{\Sigma}$ are zero at a countably infinite set of points.  Thus, the arguments given above can be applied one by one to each such point.  Furthermore, we may use the opposite pointing normal to obtain the same results for the opposing domain of discontinuity.  With this in mind, we can refine the above theorem in the following way.

\begin{theorem}\label{thm:flats}
Let $\Gamma$ be a normalized almost-Fuchsian group and $\widetilde{\Sigma}$ the almost-Fuchsian $\Gamma$-invariant disk with data $(g,\alpha)\in\mathcal{AF}.$  For a fixed compact set $E\subset \mathbb{H}^3$ containing $p=(0,0,1),$ consider any set of points $\{p_i\}\subset\widetilde{\Sigma}$ contained in $E\cap\widetilde{\Sigma}$ such that $\lVert \alpha \rVert_{g}(p_i)=0$ for each $i.$  Then there exists an $R'(E)>0$ such that $\displaystyle\bigcup_{i} B_{\mathbb{S}^2}\left(\mathcal{G}_{\widetilde{\Sigma}}^{\pm}(p_i),R'\right)\subset \Omega$ where $\Omega$ is the domain of discontinuity of $\Gamma.$
\end{theorem}

\textbf{Remark:} By Theorem \ref{thm:eps}, $\widetilde{\Sigma}$ is properly embedded so that $E\cap \widetilde{\Sigma}$ is compact. As any set of zeros $\{p_i\}$ is necessarily discrete (they are zeros of a holomorphic quadratic differential), any compact set of zeros is finite.

\begin{proof}
Firstly, we claim there exists a positive integer $K(E)$ such that the cardinality of the set of zeros of $\alpha$ contained in $\widetilde{\Sigma}\cap E$ is always less than $K.$  As observed in the remark above, $\widetilde{\Sigma}\cap E$ is compact.  Because the metric on the minimal surface is uniformly comparable to the hyperbolic metric, there exists a uniform constant $C>0$ such that the area of $\widetilde{\Sigma}\cap E$ is bounded above by $C.$  By the Gauss equation the area of the closed minimal surface which $\widetilde{\Sigma}$ covers is at least $\pi(2g-2)$ where $g$ is the genus, thus a fundamental domain for the action of the almost-Fuchsian group has area at least $\pi(2g-2).$  Then there are at most $\frac{C}{\pi(2g-2)}$ fundamental domains which lie in $\widetilde{\Sigma}\cap E.$  Since each fundamental domain contains at most $4g-4$ zeros of $\alpha,$ we take $K(E)$ equal to the nearest integer greater than $(4g-4)\frac{C}{\pi(2g-2)}=\frac{2C}{\pi}.$

For any set of zeros $\{p_i\}_{i=1}^{n}\subset \widetilde{\Sigma}\cap E,$ each member of a collection of normalizing elements $\{I_i\}_{i=1}^{n}\subset \mathrm{Isom}(\mathbb{H}^3)$ (i.e. $I_i$ satisfies $I_i(p_i)=p=(0,0,1)\big)$ lies in a fixed compact set in $\text{Isom}(\mathbb{H}^3).$  This follows from the fact that each $I_i$ will be a composition of a hyperbolic element of uniformly bounded translation distance (specifically bounded by the distance from $p$ to $\partial E$) and a rotation. Fixing the $R>0$ of Theorem \ref{thm:thickDD}, we set $R'$ equal to the minimum of the spherical radii of the disks $I\left(B_{\mathbb{C}}(0,R)\right)$ where $I\in\text{Isom}(\mathbb{H}^3)$ ranges over the finite set of all words of length at most $n$ in the $I_i$ and their inverses.  Remember that we first established that there exists a universal constant such that $n<K(E)$ and so this is a bounded list over all almost-Fuchsian disks.  Each $I_i$ is a conformal transformation so indeed $I(B_{\mathbb{C}}(0,R))$ is a disk.  Then, the argument in Theorem \ref{thm:thickDD} implies that
$\displaystyle\bigcup_{i} B_{\mathbb{S}^2}\left(\mathcal{G}_{\widetilde{\Sigma}}^{+}(p_i),R'\right)\subset \Omega.$

Repeating the same argument with the opposite pointing normal proves that there exists $R''>0$ such that $\displaystyle\bigcup_{i} B_{\mathbb{S}^2}\left(\mathcal{G}_{\widetilde{\Sigma}}^{-}(p_i),R''\right)\subset \Omega.$  Letting $R'=\min\{R',R''\}$ completes the proof.
\end{proof}

\subsection{Geometric limits of almost-Fuchsian groups}\label{subsec: limits}

In this section, we prove that a geometrically convergent sequence of almost-Fuchsian groups can not limit to a doubly-degenerate group.  The mantra is that the disks guaranteed to exist in each domain of discontinuity form barriers into which the limit set can not penetrate.

\begin{theorem}
Suppose $\Gamma_n$ is a sequence of almost-Fuchsian groups and $\Gamma_n\rightarrow \Gamma$ geometrically.  Then $\Gamma$ is not doubly-degenerate.
\end{theorem}

\begin{proof}
Recall that the limit set of a doubly-degenerate group $\Gamma$ is equal to $\partial_{\infty}(\mathbb{H}^3).$  Suppose $\Gamma_n$ is a sequence of almost-Fuchsian groups converging geometrically to $\Gamma.$  By theorem \ref{thm:thickDD} there exists an $R>0$ and a sequence of transformations $I_n\in \mathrm{Isom}(\mathbb{H}^3)$ such that 
\begin{align}\label{eqn:avoid}
\Lambda(I_n \Gamma_n I_n^{-1})\cap B_{\mathbb{C}}(0,R) =\emptyset.
\end{align}
We shall refer to any such sequence $I_n$ as a normalizing sequence.

First suppose that the $I_n$ remain in a compact subset of $\mathrm{Isom}(\mathbb{H}^3).$  Choose a subsequence such that $I_n\rightarrow I$ where we have relabeled the indices.  Then  
$I_n \Gamma_n I_n^{-1}\rightarrow I \Gamma I^{-1}$ geometrically.  Since $\Lambda(I \Gamma I^{-1})=I \Lambda(\Gamma),$ the limit set of $I \Gamma I^{-1}$ is also equal to $\partial_{\infty}(\mathbb{H}^3).$  We conclude that 
\[\partial_{\infty}(\mathbb{H}^3)=\Lambda(I \Gamma I^{-1})\subset\lim\limits_{n\rightarrow\infty}\Lambda(I_n \Gamma_n I_n^{-1})\] 
and so 
\[\lim\limits_{n\rightarrow\infty}\Lambda(I_n \Gamma_nI_n^{-1})=\partial_{\infty}(\mathbb{H}^3)\] 
in the Hausdorff topology on closed subsets of $\partial_{\infty}(\mathbb{H}^3).$  In particular, since the spherical and Euclidean metrics are comparable on compact subsets in the Euclidean topology (they nearly agree around $0\in\mathbb{C}),$ given any $R>0$ there exists an $N\in\mathbb{N}$ such that 
\[\Lambda(I_n \Gamma_n I_n^{-1})\cap B_{\mathbb{C}}(0,R)\not=\emptyset\]
for every $n\geq N.$  This contradicts equation \eqref{eqn:avoid}.

Now suppose that no normalizing sequence $I_n$ remains in a compact subset of $\mathrm{Isom}(\mathbb{H}^3).$  We work towards a contradiction.  Let $\widetilde{\Sigma_n}$ be the unique $\Gamma_n$-invariant almost-Fuchsian disk.  We select a sequence of points $p_n\in\widetilde{\Sigma_n}$ satisfying the following properties,
\begin{enumerate}
\item Each $p_n\in \widetilde{\Sigma_n}$ is such that the principal curvatures of $\widetilde{\Sigma_n}$ vanish at $p_n.$ 
\item There exists $x\in\mathbb{C}\subset\partial_{\infty}(\mathbb{H}^3)$ such that, after perhaps passing to a subsequence, $p_n\rightarrow x$ where the convergence is in the Euclidean topology on the closed upper half-space $\mathbb{H}^3\cup\partial_{\infty}(\mathbb{H}^3).$  \label{prop:bddpara}
\end{enumerate}
A (sub-)sequence can be found satisfying both conditions since we have assumed every sequence $p_n$ of vanishing principal curvature leaves every compact subset of $\mathbb{H}^3.$  Next, conjugate each group $\Gamma_n$ by a parabolic isometry, represented by a
M\"{o}bius transformation
\[J_n=z+a_n,\]  
which maps $p_n$ to some point $(0,0,p_n)\in\mathbb{H}^3$ where we have abused notation in calling the z-coordinate by the same name.  Note that assumption \eqref{prop:bddpara} on the sequence $p_n$ implies $\lvert a_n \rvert \leq K$ for some $K>0.$  As the next step, conjugate each group further by an elliptic isometry to make the downward pointing unit normal vector to $J_n(p_n)$ directed at $0\in\partial_{\infty}(\mathbb{H}^3).$  Further abusing notation, we label the new sequence of groups as $\Gamma_n$ and denote the unique $\Gamma_n$-invariant almost-Fuchsian disks by $\widetilde{\Sigma_n}.$  

It remains true that $\Lambda(\Gamma_n)\rightarrow \partial_{\infty}(\mathbb{H}^3)$ since the parabolic elements were chosen from a compact set of isometries and elliptic elements act as isometries of the spherical metric from which the Hausdorff topology was induced.  

Given the above prerequisites, a $\widetilde{\Sigma_n}$-normalizing sequence $I_n$ takes the form
\begin{align}
I_n =
 \begin{pmatrix}
  \lambda_n & 0 \\
  0 & \lambda_n^{-1}
 \end{pmatrix}
\in \mathrm{PSL}(2,\mathbb{R})\subset \mathrm{Isom}^{+}(\mathbb{H}^3)
\end{align}
for some $\lambda_n \rightarrow \infty.$

Since $I_n(\widetilde{\Sigma_n})$ is a sequence of normalized almost-Fuchsian disks, by Theorem \ref{thm:flats} there exists an $R'>0$ such that 
\begin{align}\label{eqn:view}
\Lambda(I_n\Gamma_n I_n^{-1})\cap  B_{\mathbb{S}^2}(\infty,R')=\emptyset.
\end{align}
Since $\lambda_n\rightarrow \infty,$ there exists an $N\in\mathbb{N}$ such that 
\begin{align}\label{eqn:roomy}
\partial_{\infty}(\mathbb{H}^3)\setminus I_n^{-1}(B_{\mathbb{S}^2}(\infty,R'))&=(\mathbb{C}\cup\{\infty\})\setminus I_n^{-1}(B_{\mathbb{S}^2}(\infty,R')) \\
&\subset B_{\mathbb{S}^2}(0,R').
\end{align}
for all $n>N.$
Applying $I_n^{-1}$ to equation \eqref{eqn:view} we obtain that,
\begin{align}
\Lambda(\Gamma_n)\cap I_n^{-1} B_{\mathbb{S}^2}(\infty,R')=\emptyset.
\end{align}
Note that the above is true because $I_n^{-1}$ is injective, so it maps non-intersecting sets to non-intersectings sets.
By line \eqref{eqn:roomy} this implies that for all $n>N$
\begin{align}
\Lambda(\Gamma_n)\subset B_{\mathbb{S}^2}(0,R'),
\end{align}
but this is impossible since $\Lambda(\Gamma_n)\rightarrow \partial_{\infty}(\mathbb{H}^3).$  This contradiction implies that we may always find a normalizing sequence which remains in a compact subset of $\mathrm{Isom}(\mathbb{H}^3).$  But, we have already shown above that in this case, hence $\Gamma_n$ can not geometrically limit to a doubly-degenerate group.  This completes the proof.
\end{proof}

\bibliography{DDBib}{}

\begin{thebibliography}{GHW10}

\bibitem[AG85]{AG85}
K.~Astela and F.~W. Gehret.
\newblock Quasiconformal analogues of theorems of {K}oebe and
  {H}ardy-{L}ittlewood.
\newblock {\em Michigan Math. J.}, \textbf{32}(1):99--107, 1985.

\bibitem[BCM12]{BCM12}
Jeffrey~F. Brock, Richard~D. Canary, and Yair~N. Minsky.
\newblock The classification of {K}leinian surface groups, {II}: {T}he ending
  lamination conjecture.
\newblock {\em Ann. of Math. (2)}, 176(1):1--149, 2012.

\bibitem[Ber60]{BER60}
Lipman Bers.
\newblock Simultaneous uniformization.
\newblock {\em Bull. Amer. Math. Soc}, \textbf{66}:94--97, 1960.

\bibitem[Bon86]{BON86}
Francis Bonahon.
\newblock Bouts de varieties hyperboliques de dimension 3.
\newblock {\em Annals of Math.}, \textbf{2}(124):71--158, 1986.

\bibitem[CEG87]{CEG87}
Richard Canary, David Epstein, and P.~Green.
\newblock Notes on notes of {T}hurston.
\newblock In {\em Analytical and geometric aspects of hyperbolic space}, pages
  3--92. Cambridge University Press, 1987.

\bibitem[Eps84]{EPS84}
Charles~L. Epstein.
\newblock Envelopes of horospheres and weingarten surfaces in hyperbolic
  3-space.
\newblock Unpublished, 1984.

\bibitem[Eps86]{EPS86}
Charles~L. Epstein.
\newblock The hyperbolic {G}auss map and quasiconformal reflections.
\newblock {\em J. Reine Angew. Math.}, \textbf{372}:96--135, 1986.

\bibitem[FHS83]{FHS83}
Michael Freedman, Joel Hass, and Peter Scott.
\newblock Least area incompressible surfaces in 3-manifolds.
\newblock {\em Invent. Math.}, \textbf{71}(3):609--642, 1983.

\bibitem[GHW10]{GHW10}
Ren Guo, Zheng Huang, and Biao Wang.
\newblock Quasi-{F}uchsian three-manifolds and metrics on {T}eichmuller space.
\newblock {\em Asian J. Math.}, \textbf{14}(2):243--256, 2010.

\bibitem[Hop54]{HOP54}
Heinz Hopf.
\newblock {\em Differential geometry in the large}, volume 1000 of {\em Lecture
  Notes in Mathematics}.
\newblock Springer-Verlag, Berlin, 1954.

\bibitem[HW13]{HW13}
Zheng Huang and Biao Wang.
\newblock On almost-{F}uchsian manifolds.
\newblock {\em Trans. Amer. Math. Soc.}, 365(9):4679--4698, 2013.

\bibitem[Jor77]{JOR77}
Troels Jorgensen.
\newblock Compact 3-manifolds of constant negative curvature fibering over the
  circle.
\newblock {\em Annals of Math.}, \text{106}(1):61--72, 1977.

\bibitem[Min10]{MIN10}
Yair Minsky.
\newblock The classification of {K}leinian surface groups {I}: {M}odels and
  bounds.
\newblock {\em Annals of Math.}, \textbf{171}(1):1--107, 2010.

\bibitem[SU82]{SU82}
J.~Sacks and K.~Uhlenbeck.
\newblock Minimal immersions of closed {R}iemann surfaces.
\newblock {\em Amer. Math. Soc.}, \textbf{271}(2):639--652, 1982.

\bibitem[SY79]{SY79}
Richard Schoen and Sing-Tung Yau.
\newblock Existence of incompressible minimal surfaces and the topology of
  three-dimensional manifolds with nonnegative scalar curvature.
\newblock {\em Ann. of Math.}, \textbf{110}(1):127--142, 1979.

\bibitem[Thu]{THU}
William Thurston.
\newblock Hyperbolic structures on 3-manifolds, {II}: surface groups and
  manifolds which fiber over the circle.
\newblock {\em Preprint}.
\newblock arXiv:math.GT/9801045.

\bibitem[Uhl83]{UHL83}
Karen~K. Uhlenbeck.
\newblock {\em Closed minimal surfaces in hyperbolic 3-manifolds}, volume
  \textbf{103}.
\newblock Princeton {U}niv. {P}ress, Princeton, {NJ}, 1983.

\end{thebibliography}
\bibliographystyle{alpha}

\end{document}